\documentclass[12pt,a4paper]{article}
\usepackage{graphics}
\usepackage{graphicx}
\usepackage{caption}
\usepackage{floatrow}
\usepackage{float}
\usepackage{pifont}
\usepackage{subcaption}
\usepackage{epsfig}
\usepackage[mathscr]{euscript}
\usepackage{mathrsfs}
\usepackage{amsmath,amsthm,amsfonts}
\usepackage{amssymb}
\usepackage{enumerate}
\usepackage{hyperref}
\usepackage{blkarray}
\usepackage{multirow}

\newtheorem{theorem}{Theorem}[section]

\newtheorem{proposition}[theorem]{Proposition}

\newtheorem{lemma}[theorem]{Lemma}

\begin{document}
\title{There does not exist a distance-regular graph with intersection array  $\{80, 54,12; 1, 6, 60\}$ }

\author{Jack H. Koolen$^{\dag\S}$, Quaid Iqbal$^\dag$, Jongyook Park$^\ast$, Masood Ur Rehman$^\dag$
\\ \\
\small $^\dag$School of Mathematical Sciences,\\
\small University of Science and Technology of China, \\
\small 96 Jinzhai Road, Hefei, 230026, Anhui, PR China\\
\small $^\S$Wen-Tsun Wu Key Laboratory of CAS,\\
\small 96 Jinzhai Road, Hefei, 230026, Anhui, PR China\\
\small $^\ast$Department of Applied Mathematics,\\
\small Wonkwang University, Iksan, Jeonbuk, 54538,  Republic of Korea\\
\small {\tt e-mail: koolen@ustc.edu.cn, quaid@mail.ustc.edu.cn,} \\
\small {\tt~~~~~~~ jongyook@wku.ac.kr, masood@mail.ustc.edu.cn}\vspace{-3pt}
}
\maketitle
\date{}
\vspace{-24pt}
\begin{abstract}
In this paper we will show that there does not exist a distance-regular graph $\Gamma$ with intersection array  $\{80, 54,12; 1, 6, 60\}$. We first show that  a local graph $\Delta$ of $\Gamma$ does not contain a coclique with 5 vertices, and then we prove that the graph $\Gamma$ is geometric by showing that $\Delta$ consists of 4 disjoint cliques with 20 vertices. Then we apply a result of Koolen and Bang to the graph $\Gamma$, and we could obtain that there is no such a distance-regular graph.
\end{abstract}



\section{Introduction}
In this paper we will show that there does not exist a distance-regular graph $\Gamma$ with intersection array  $\{80, 54,12; 1, 6, 60\}$. In order to do so, we will first study the claw-bound that was introduced by Bose in the case of pseudo geometric strongly regular graphs \cite{Bose63}. We note that a similar bound for amply-regular graphs with an $s$-claw was shown by Godsil~\cite{CG1993} and that Koolen and Park~\cite{KPShilla} also studied the claw-bound for distance-regular graphs and they mentioned that three known examples of Terwilliger graphs with $c_2\geq2$ meet equality in the claw-bound. After that, Gavrilyuk~\cite{ALG2010} showed that if an amply-regular Terwilliger graph with $c\geq2$ attains equality in the claw-bound, then the graph is one of the known Terwilliger graphs.

Next, we will show that for a given vertex $x$ of $\Gamma$, the subgraph $\Delta$ induced on the neighbors of $x$ does not contain a coclique of size $5$. From the claw-bound, it is easy to see that $\Delta$ does not contain a coclique of size $6$. One may think that non-existence of a coclique of size $5$ follows from the result of Gavrilyuk~\cite[Theorem 4.2]{ALG2010}, but it is not true because his result works if for any non-adjacent two distinct vertices in $\Delta$, there exists a coclique of size $5$ containg them. So, we need to give a proof to show that $\Delta$ does not contain a coclique of size $5$ without using the result of Gavrilyuk.  Also, we remark that Gavrilyuk~\cite[Theorem 4.2]{ALG2010} missed to state that the graph is Terwilliger. 

Finally, we will show that $\Delta$ consists of $4$ disjoint cliques of size $20$. This shows that the graph $\Gamma$ is geometric and then a result by Koolen and Bang~\cite{KB2010} finishes the proof. For undefined terminologies, see Section 2 or \cite{BCN89,VKT2016}.

\section{Preliminaries}
All the graphs considered in this paper are finite, undirected and simple. The reader is referred to \cite{BCN89} for more information. Let $\Gamma$ be a connected graph with vertex set $V(\Gamma)$. The {\em distance} $d_{\Gamma}(x,y)$ between two vertices $x,y\in V(\Gamma)$ is the length of a shortest path between $x$ and $y$ in $\Gamma$. The {\em diameter} $D=D(\Gamma)$ of $\Gamma$ is the maximum distance between
any two vertices of $\Gamma$. For each $x\in V(\Gamma)$, let $\Gamma_i(x)$ be the set of vertices in $\Gamma$ at distance
$i$ from $x$ ($0\leq i\leq D$). In addition, define $\Gamma_{-1}(x)=\emptyset$ and $\Gamma_{D+1}(x)=\emptyset$. For the sake of simplicity, let $\Gamma(x)=\Gamma_1(x)$ and we denote $x\sim y$ if two vertices $x$ and $y$ are adjacent. In particular, $\Gamma$ is {\em regular} with {\em valency} $k$ if $k=|\Gamma(x)|$ holds for all $x\in V(\Gamma)$. For a vertex $x$ of $\Gamma$, the number $|\Gamma(x)|$ is called the valency of $x$ in $\Gamma$.

For any two vertices $x,y$ at distance $i=d_{\Gamma}(x,y)$, we consider the numbers $c_i(x,y)=|\Gamma_{i-1}(x)\cap \Gamma(y)|$, $a_i(x,y)=|\Gamma_i(x)\cap \Gamma(y)|$ and $b_i(x,y)=|\Gamma_{i+1}(x)\cap \Gamma(y)|$ $(0\leq i\leq D)$. When $d_{\Gamma}(x,y)=2$, we usually denote $c_2(x,y)$ by $\mu_{\Gamma}(x, y)$.

We say that {\em intersection number} $a_i$ ($b_i$ and $c_i$, respectively) exists if the number $a_i(x,y)$ ($b_i(x,y)$ and $c_i(x,y)$, respectively) does depend only on $i=d_{\Gamma}(x,y)$ not on the choice of $(x,y)$ with $d_{\Gamma}(x,y) =i$. Set $c_0 = b_D = 0$ and observe $a_0=0$ and $c_1=1$. A connected graph $\Gamma$ with diameter $D$ is called a {\em  distance-regular graph} if there exist intersection numbers $c_i, a_i, b_i$ for all $i=0,1,\ldots, D$. The {\em intersection array} of a distance-regular graph with diameter $D$ is the array $\{b_0,b_1,\ldots,b_{D-1};c_1,c_2,\ldots,c_D\}$.

For a connected graph $\Gamma$ with diameter $D$, the {\em adjacency matrix} $A=A(\Gamma)$ is the matrix whose rows
and columns are indexed by $V(\Gamma)$, where the ($x, y$)-entry is $1$ whenever $x\sim y$ and
$0$ otherwise. The {\em eigenvalues} of $\Gamma$ are the eigenvalues of $A(\Gamma)$.

A {\em clique} in a graph is a set of pairwise adjacent vertices. A {\em complete} graph is a graph whose vertex set is a clique. And a {\em coclique} in a graph is a set of pairwise non-adjacent vertices. A graph with vertex set as a clique is called a {\em complete graph}. For a distance-regular graph $\Gamma$ with valency $k$, diameter $D\geq2$ and the smallest eigenvalue $\theta_D$, it is known that the size of a clique $C$ in $\Gamma$ is bounded by $|C|\leq1-k/\theta_D$ (\cite{BCN89}). A clique $C$ in $\Gamma$ is called a {\em Delsarte clique} if $C$ contains exactly $1-k/\theta_D$ vertices. A distance-regular graph $\Gamma$ with diameter $D\geq2$ is called {\em geometric} if there exists a set $\mathcal{C}$ of Delsarte cliques such that each edge of $\Gamma$ lies in a unique $C\in \mathcal{C}$.

For a distance-regular graph $\Gamma$ with diameter $D$, let $C,C'$ be non-empty subsets of $V(\Gamma)$ and $x$ be a vertex of $\Gamma$. We write  $d_{\Gamma}(x,C):={\rm min}\{d_{\Gamma}(x,y)~|~y\in C\}$ and $d_{\Gamma}(C,C'):={\rm min}\{d_{\Gamma}(x,y)~|~x\in C,~ y\in C'\}$. The {\em covering radius} of $C$, denoted by $\rho(C)$ is defined as $\rho(C):={\rm max}\{d_{\Gamma}(x,C)~|~x\in V(\Gamma)\}$, and define $C_i:=\{x\in V(\Gamma)~|~ d_{\Gamma}(x,C)=i\}$ $(0\leq i\leq \rho(C))$. We write $B_{xi}(C):=|C\cap \Gamma_i(x)|$ and the numbers $B_{xi}(C)$ $(i=0,1,\ldots,D)$ are called the {\em outer distribution numbers} of $C$. A non-empty subset $C$ of $V(\Gamma)$ with covering radius $\rho(C)$ is called a {\em complete regular code}, if the outer distribution number $B_{xi}(C)$ is dependent only on $i$ and $d_{\Gamma}(x,C)$, i.e., there exist numbers $e_{\ell i}$ $(\ell=0,1,\ldots,\rho(C), i=0,1,\ldots,D)$ such that  for all vertices $x$ of $\Gamma$ and $i\in\{0,1,\ldots,D\}$, we have $B_{xi}(C)=e_{\ell i}$, where $\ell=d_{\Gamma}(x,C)$. We write $\psi_i(C):=e_{ii}$ for $i=0,1,\ldots,\rho(C)$. Assume that $\Gamma$ is a geometric distance-regular graph with a set $\mathcal{C}$ of Delsarte cliques and diameter $D$. For a given interger $j\in\{0,1,\ldots,D\}$ and any pair of vertices $x,y\in V(\Gamma)$ with $d_{\Gamma}(x,y)=j$, we denote the number of cliques $C\in\mathcal{C}$ that contain $y$ and satisfy $d_{\Gamma}(x,C)=j-1$ by $\tau_j:=\tau_j(\mathcal{C})$.

In 2010, Koolen and Bang~\cite{KB2010} showed the following two results, and they are important to prove Theorem~\ref{main}.

\begin{lemma}(cf. \cite[Lemma 4.1]{KB2010})\label{taupsi}
Suppose that $\Gamma$ is a geometric distance-regular graph with valency $k \geq 2$, diameter $D \geq 2$ and smallest eigenvalue $\theta_D$ . Then the following hold:
\begin{enumerate}
\item[$(i)$] $b_i=-(\theta_D+\tau_i)(1-k/\theta_D-\psi_i)$  $(1\leq i\leq D-1)$.
\item[$(ii)$] $c_i=\tau_i\psi_{i-1}$ $(1\leq i \leq D)$.
\end{enumerate}
\end{lemma}

\begin{lemma}(cf. \cite[Lemma 4.2]{KB2010})\label{kbgeometric}
Let $\Gamma$ be a geometric distance-regular graph with diameter $D\geq2$ and intersection number $c_2\geq2$. Then $\tau_2\geq \psi_1$ holds.
\end{lemma}

\section{The claw-bound}
In this section we recall the claw-bound which was found by several authors. For example, Bose introduced it for the class of pseudo geometric strongly regular graphs \cite{Bose63}, Godsil~\cite{CG1993} showed it for amply-regular graphs, Koolen and Park~\cite{KPShilla} worked on the claw-bound for distance-regular graphs and Gavrilyuk~\cite{ALG2010} gave further results on the claw-bound for amply regular graphs. In the following lemma, we will reprove the claw-bound because the equality case in the claw-bound is important in this paper.

\begin{lemma}\label{clawlemma}
Let $\Gamma$ be a regular graph with valency $k$ and $n$ vertices.
Assume that there exists a constant $c$ such that $\mu_{\Gamma}(x,y) \leq c$ for all distinct non-adjacent vertices $x$ and $y$ of $\Gamma$.
Also let $\bar{C}$ be a coclique in $\Gamma$ with $s$ vertices. 
Then 
\begin{equation}\label{claw}
\frac{(k+1)s-n}{{s \choose 2}}\leq c
\end{equation}
holds.
Moreover, if equality holds in $(\ref{claw})$, then any vertex outside $\bar{C}$ has at least one neighbor in $\bar{C}$ and any vertex of $\Gamma$ has at most two neighbors in $\bar{C}$, and $\mu_{\Gamma}(y, y') = c$ for all distinct vertices $y$ and $y'$ of $\bar{C}$. In particular, if equality holds for $s$ in  $(\ref{claw})$, then all cocliques in $\Gamma$ have at most $s$ vertices.
\end{lemma}
\begin{proof}
Assume that $\bar{C}$ has vertives $y_1, y_2, \ldots,y_s$. We note that $|\Gamma(y_i)\cap\Gamma(y_j)|\leq c$ for all $1\leq i< j\leq s$. Then by the principle of inclusion and exclusion, we have

$n$ $\geq$ $|\displaystyle\cup_{i=1}^{s}(\Gamma(y_i)\cup\{y_i\})|$ $\geq$ $\displaystyle\sum_{i=1}^s|\Gamma(y_i)\cup\{y_i\}|$ $-$ $\displaystyle\sum_{1\leq i< j\leq s}|\Gamma(y_i)\cap\Gamma(y_j)|$.

This shows that $n\geq s(k+1)-{s\choose 2}c$, and equality implies that any vertex outside $\bar{C}$ has at least one neighbor in $\bar{C}$ and any vertex of $\Gamma$ has at most two neighbors in $\bar{C}$, and $\mu_{\Gamma}(y_i, y_j) =|\Gamma(y_i)\cap\Gamma(y_j)|= c$ for all $1\leq i< j\leq s$.
\end{proof}

A connected graph $\Gamma$ is called {\em amply regular} with parameters $(k, a, c)$ if $\Gamma$ is a regular graph with valency $k$ satisfying that any pair of adjacent vertices of $\Gamma$ have exactly $a$ common neighbors and any two vertices of $\Gamma$ at distance $2$ have $c$ common neighbors. 
For a vertex $x$ of a graph $\Gamma$, the {\em local graph} $\Delta = \Delta(x)$ at $x$ is the subgraph induced on the neighbors of $x$. Note that for a vertex $x$ of  an amply regular graph $\Gamma$ with parameters $(k, a, c)$, any maximal coclique in the local graph $\Delta(x)$ has at least  $\frac{k}{a+1}$ vertices. 

By applying  Lemma \ref{clawlemma} to local graphs of amply regular graphs we obtain the following proposition. 
\begin{proposition}\label{clawprop}
Let $\Gamma$ be an amply regular graph with parameters $(k, a, c)$. Let $x$ be a vertex of $\Gamma$ and let $C$ be a maximal coclique of size $s$ in the local graph $\Delta(x)$. 
Then \begin{equation}\label{claw2}
\frac{(a+1)s-k}{{s \choose 2}}\leq c-1
\end{equation}
holds.
Moreover, if equality holds in $(\ref{claw2})$, then  any vertex $z$ of $\Delta(x)$ which is not in $C$, has at least one neighbor and at most two neighbors in $C$, and $\mu_{\Delta}(y_1, y_2) = c-1$ for all distinct vertices $y_1,y_2\in\bar{C}$.
\end{proposition} 
\begin{proof}
This result immediately follows from  Lemma \ref{clawlemma}.
\end{proof}

Gavrilyuk~\cite{ALG2010} showed the following:

\begin{theorem} Let $\Gamma$ be an amply regular graph with parameters $(k, a, c)$. Assume that there exists a positive integer $s$ such that equality holds in $(2)$ and let $s'$ be the smallest such $s$. If every induced path of length $2$ of $\Gamma$  is a subgraph of an induced subgraph $K_{1,s'}$ of $\Gamma$,
then $\Gamma$ is the icosahedron, the Doro graph or the Conway-Smith graph. 
\end{theorem}

Note that in \cite[Theorem 4.2]{ALG2010}  the assumption that the graph has to be Terwilliger is missing.

\section{No coclique with $5$ vertices in a local graph}

In this section, we will show that any local graph $\Delta$ of a distance-regular graph $\Gamma$ with intersection array $\{80,54,12; 1,6,60\}$ does not contain a coclique with $5$ vertices. Note that $\Gamma$ has eigenvalues $80, 26, 5$ and $-4$ with multiplicities $1, 80, 144$ and $720$ respectively. Also we note that $\Delta$ is a regular graph with $80$ vertices and valency $25$. 

In the following lemma, we will study some properties of maximal cliques in a local graph $\Delta$ of $\Gamma$.

\begin{lemma}\label{clique}
Let $\Gamma$ be a distance-regular graph with intersection array \\
$\{80,54,12; 1,6,60\}$. Let $x$ be a vertex of $\Gamma$ and let $\Delta:=\Delta(x)$ be the local graph at $x$.
Let $C_1$ and $C_2$ be two maximal cliques in $\Delta$ with $v_1$ and $v_2$ vertices respectively.
Then the following hold:
\begin{enumerate}
\item[$(i)$] The smallest eigenvalue of $\Delta$ is at least $-3.$
\item[$(ii)$] $v_1 \leq 20$ and if equality holds, then every vertex $z$ of $\Delta$ which is not in $C_1$ has exactly $2$ neighbors in $C_1$. 
\item[$(iii)$] If $v_1$ and $v_2$ are both at least $11$, then $C_1$ and $C_2$ intersect in at most $4$ vertices.
\item[$(iv)$] If $v_1 + v_2 \geq 31$, then $C_1$ and $C_2$ do not intersect.
\end{enumerate}
\end{lemma}
\begin{proof}
$(i)$: It follows by Theorem  \cite[Theorem 4.4.3]{BCN89}, that the smallest eigenvalue of $\Delta$ is at least $-1-\frac{b_{1}}{(\theta_{1}+1)}=-1-\frac{54}{27}=-3$, where $\theta_1$ is the second largest eigenvalue of $\Gamma$.
\newline
$(ii)$: By \cite[Proposition 4.4.6]{BCN89}, any clique $C$ of the graph $\Gamma$ has at most $1 + \frac{80}{4} = 21$ vertices and if equality holds, then any vertex adjacent to $C$ has exactly $3$ neighbors in $C$. This shows $(ii)$. 
\newline
$(iii)$: Without loss of generality, we assume that $11\leq v_1 \leq v_2 \leq 20$. And we assume that $C_1$ and $C_2$ intersect in at least 5 vertices. As $c_2=6$, $C_1$ and $C_2$ intersect in exactly 5 vertices. We denote the set of the five commnon vertices of $C_1$ and $C_2$ by $I$. Let $\pi=\{(C_1\cup C_2)\setminus I,I\}$ be a partition of the union of $C_1$ and $C_2$. And we consider the quotient matrix $Q$ corresponding to $\pi$.
Then $Q$ is of the form 
$Q= \left( \begin{matrix}
4 & v_1+v_2-10\\
5 & \frac{(v_1-5)(v_1-6)+(v_2-5)(v_2-6)}{v_1+v_2-10}
\end{matrix} \right)$. Note that the smallest eigenvalue of $Q$ is less than $-3$ for all $11\leq v_1 \leq v_2 \leq 20$. By interlacing, the smallest eigenvalue of $\Delta$ is also less then $-3$. This contradicts $(i)$. So, $C_1$ and $C_2$ intersect in at most $4$ vertices.
\newline
$(iv)$: As $v_1 + v_2 \geq 31$, $v_1$ and $v_2$ are both at least $11$. By $(iii)$, we know that  $C_1$ and $C_2$ do intersect in at most 4 vertices. Suppose that $C_1$ and $C_2$ have a common vertex, say $x$. Then $x$ has at least $3 + (31-8) = 26$ neighbors in $\Delta$. This contradicts that the valency of $\Delta$ is $25$. So, $C_1$ and $C_2$ have no common vertex.
\end{proof}

We will show that any coclique in a local graph $\Delta$ of $\Gamma$ has at most $4$ vertices.  From Proposition~\ref{clawlemma}, one can easily see that  $\Delta$ does not contain a coclique with $6$ vertices. Now we will first study  the following lemmas, and then in Proposition~\ref{5coclique}, we will prove that $\Delta$ does not contain a coclique with $5$ vertices.

For vertices $x_{1},\dots, x_{s}$ of $\Delta$, we define the  graphs $ \bar{\Delta}(x_{1},\dots, x_{s})$ as the subgraph of $\Delta$ induced  on the vertices of $\Delta$ that are not adjacent to all of $x_{1},\dots, x_{s}$.

\begin{lemma}\label{5coclique1} Let $\Gamma$ be a distance-regular graph with intersection array \\
$\{80,54,12; 1,6,60\}$. Let $x$ be a vertex of $\Gamma$ and let $\Delta:=\Delta(x)$ be the local graph at $x$. If $\Delta$ contains a coclique $\bar{C}$ with vertices $y_1,y_2,y_3,y_4,y_5$, then the following hold:
\begin{enumerate}
\item[$(i)$] $\mu_{\Delta}(y_1, y_2) = 5$;
\item[$(ii)$] Every vertex of $\Delta$ is adjacent to at most two vertices of $\bar{C}$;
\item[$(iii)$] $\bar{\Delta}(y_1, y_2, y_3, y_4)$ is a complete graph with $6$ vertices containing $y_5$;
\item[$(iv)$] $\bar{\Delta}(y_1, y_2, y_3)$ is a graph on $17$ vertices with valencies $10$ and possibly $16$.
\item[$(v)$] $\bar{\Delta}(y_1, y_2)$ is a graph on $33$ vertices and any vertex of $\bar{\Delta}(y_1, y_2)$ 
has valency at least $15$. 
Moreover, it contains a vertex with even valency.
\end{enumerate}
\end{lemma}

\begin{proof} As $\Delta$ has valency $25$ and $80$ vertices, we have 
$$\frac{5(25+1) -80 }{{5 \choose 2}} = 5 = c_2 -1.$$
By Proposition~\ref{clawprop}, statements $(i)$ and $(ii)$ follow.\\
Note that $\bar{\Delta}(y_1, y_2, y_3, y_4) $ has $25 +1 - 4 \times5= 6$ vertices. As $\Delta$ does not contain a coclique with 6 vertices,  $\bar{\Delta}(y_1, y_2, y_3, y_4) $ is a complete graph. This shows $(iii)$.\\
We note that, by $(i)$ and $(ii)$, the graph $\bar{\Delta}(y_1, y_2, y_3)$ has $80 - 3\times(25+1) +
3\times5 = 17$ vertices.

Let $z_1$ be a vertex of $\bar{\Delta}(y_1, y_2, y_3)$. If $z_1$ is adjacent to all other vertices in $\bar{\Delta}(y_1, y_2, y_3)$, then the valency of $z_1$ in $\bar{\Delta}(y_1, y_2, y_3)$ is $16$. So let us assume that there exists a distinct vertex $z_2$ not a neighbor of $z_1$ in $\bar{\Delta}(y_1, y_2, y_3)$. Then $\{y_1, y_2, y_3, z_1, z_2\}$ is a coclique with $5$ vertices in $\Delta$. By $(i)$ and $(ii)$,  both $z_1$ and $z_2$ have valency $10$ inside $\bar{\Delta}(y_1, y_2, y_3)$. This shows $(iv)$.\\
In a similar manner as in $(iv)$, we see that $\bar{\Delta}(y_1, y_2)$ has $80 - 2\times(25+1) +5 = 33$ vertices.
By $(i)$ and $(ii)$, it follows that any vertex of $\bar{\Delta}(y_1, y_2)$ has valency at least $25-2\times5 =15$.
As the number of vertices of $\bar{\Delta}(y_1, y_2)$ is odd, we see, by the handshaking lemma, that $\bar{\Delta}(y_1, y_2)$ has a vertex with even valency. This shows $(v)$.
\end{proof}

From Lemma~\ref{5coclique1} $(v)$, we know that $\bar{\Delta}(y_1,y_2)$ has a vertex with valency at least $16$. In the following lemma, we will study properties of a vertex with valency at least $16$  in $\bar{\Delta}(y_1,y_2)$.

\begin{lemma} \label{5coclique2} Let $\Gamma$ be a distance-regular graph with intersection array \\
$\{80,54,12; 1,6,60\}$. Let $x$ be a vertex of $\Gamma$ and let $\Delta:=\Delta(x)$ be the local graph at $x$. Assume that $\Delta$ contains a coclique $\bar{C}$ with vertices $y_1,y_2,y_3,y_4,y_5$.
Let $z$ be a vertex of $\bar{\Delta}(y_1, y_2)$ with valency at least $16$ (in $\bar{\Delta}(y_1, y_2)$). Then any maximal coclique in $\Delta$ containing $y_1, y_2$ and $z$ has exactly $4$ vertices and 
$z$ is adjacent to exactly two of $y_3, y_4$ and $y_5$. Moreover, if $z$ is adjacent to $y_i$ and $y_j$ for $i\neq j\in\{3,4,5\}$, then $z$ has valency $16$ in $\bar{\Delta}(y_1, y_2, y_h)$ for $h\in \{3,4,5\}\setminus\{i,j\}$.
\end{lemma}

\begin{proof}
If there exists a maximal coclique in $\Delta$ with $5$ vertices containing $y_1$, $y_2$ and $z$, then by Lemma~\ref{5coclique1} $(i)$ the valency of $z$ in $\bar{\Delta}(y_1,y_2)$ is 15, and this contradicts that $z$ has valency at least $16$  in $\bar{\Delta}(y_1,y_2)$. So, any coclique in $\Delta$ containing $y_1,y_2$ and $z$ has at most $4$ vertices. As $26\times3<80$,  any coclique in $\Delta$ containing $y_1,y_2$ and $z$ has exactly $4$ vertices.  By Lemma~\ref{5coclique1} $(ii)$, we know that $z$ is adjacent to at most two of $y_3, y_4$ and $y_5$. If $z$ is adjacent to at most one of $y_3,y_4$ and $y_5$, then there exists a coclique in $\Delta$ with $5$ vertices containing $y_1,y_2$ and $z$, a contradiction. This shows that $z$ is adjacent to exactly two of $y_3, y_4$ and $y_5$. Without loss of generality, we assume that $z$ is adjacent to $y_4$ and $y_5$, i.e., $z$ is a vertex of $\bar{\Delta}(y_1, y_2, y_3)$. 
As $\{ y_1, y_2, y_3, z\}$ is a maximal coclique in $\Delta$ with $4$ vertices, $z$ is adjacent to all vertices of $\bar{\Delta}(y_1,y_2,y_3)$. From Lemma~\ref{5coclique1} $(iv)$, we know that $\bar{\Delta}(y_1, y_2, y_3)$ has $17$ vertices. This means that  $z$ has valency $16$ in $\bar{\Delta}(y_1, y_2, y_3)$.
\end{proof}

Assume that there exist three vertices $z_1,z_2$ and $z_3$ of $\Delta$ such that there does not exist a coclique with $5$ vertices containing all of them. Then it is easy to see that $\bar{\Delta}(z_1,z_2,z_3)$ is a complete graph. In the following lemma, we study the size of the complete graph $\bar{\Delta}(z_1,z_2,z_3)$.


\begin{lemma}\label{5coclique3}  Let $\Gamma$ be a distance-regular graph with intersection array \\
$\{80,54,12; 1,6,60\}$. Let $x$ be a vertex of $\Gamma$ and let $\Delta:=\Delta(x)$ be the local graph at $x$. Assume that there exist three vertices $z_1,z_2$ and $z_3$ of $\Delta$ such that it can not be extended to a coclique with $5$ vertices. Let $z_4$ be a vertex of $\bar{\Delta}(z_1, z_2, z_3)$ with valency $t$ in $\bar{\Delta}(z_1, z_2)$ for some integer $t$. Then $t \geq 15$ and $z_4$ lies in a clique $C$ of $\bar{\Delta}(z_1, z_2, z_3)$ with at least $t-4 \geq 11$ vertices. 
\end{lemma}
\begin{proof}
As $c_2=6$, we know that $z_4$ has at most 5 common neighbors with each of $z_1, z_2$ and $z_3$. This shows that $z_4$ has valency $t\geq15$ in $\bar{\Delta}(z_1, z_2)$ and that $\bar{\Delta}(z_1, z_2, z_3)$ has at least $1 + t -5 \geq 11$ vertices. Note that the graph $\bar{\Delta}(z_1, z_2, z_3)$ is a complete graph as $z_1, z_2, z_3$ do not lie in a coclique with $5$ vertices. This means that a clique $C$ of $\bar{\Delta}(z_1, z_2, z_3)$ containing $z_4$ has at least $t-4\geq11$ vertices.
\end{proof}

In the following proposition we will show that a local graph of the distance-regular graph with intersection array $\{80,54,12;1,6,60\}$ does not have a coclique with $5$ vertices as an induced subgraph.

\begin{proposition}\label{5coclique}
 Let $\Gamma$ be a distance-regular graph with intersection array \\
$\{80,54,12; 1,6,60\}$. Let $x$ be a vertex of $\Gamma$ and let $\Delta:=\Delta(x)$ be the local graph at $x$. Then the graph $\Delta$ does not contain a coclique with $5$ vertices.
\end{proposition}
\begin{proof}
We recall that $\Delta$ does not contain a coclique with $6$ vertices (Propsotion~\ref{clawprop}). Suppose that $\Delta$ contains a coclique with $5$ vertices $y_1,y_2,y_3,y_4$ and $y_5$. Let us consider a graph $\Sigma$ on $5$ vertices $1,2,3,4$ and $5$, where two vertices $i$ and $j$ are adjacent  if $i \neq j$ and 
 there exists a vertex of $\bar{\Delta}(y_p,y_q,y_r)$ with valency 16 for $\{ p,q,r\} = \{1,2,3,4,5\}\setminus \{i,j\}$. 

 Let $i,j$ and $r$ be distinct three vertices of $\Sigma$. From Lemma~\ref{5coclique1} $(v)$, we know that $\bar{\Delta}(y_p,y_q)$ contains a vertiex $z$ with valency at least $16$, where  $\{ p,q\} = \{1,2,3,4,5\}\setminus \{i,j,r\}$.  By Lemma~\ref{5coclique2}, we know that $z$ is adjacent to exactly two of $y_i,y_j$ and $y_r$. Without loss of generality, we may assume that $z$ is not adjacent to $y_r$. Also Lemma~\ref{5coclique2} implies that $z$ has valency $16$ in $\bar{\Delta}(y_p,y_q,y_r)$. Then two vertices $i$ and $j$ of $\Sigma$ are adjacent, i.e., there exists an edge between arbitrary chosen three vertices of $\Sigma$. This shows that $\Sigma$ does not contain a coclique with 3 vertices, and hence $\Sigma$ is not bipartite.

We will show that $\Sigma$ is triangle-free. Assume that $\Sigma$ contains a triangle, say on $\{1,2,3\}$. Then there exist three distinct vertices $z_1, z_2$ and $z_3$ of $\bar{\Delta}(y_4,y_5)$ such that each of $z_1,z_2$ and $z_3$ has valency at least $16$ and $z_i$ is adjacent to $y_j$ if  $j\in\{1,2,3\}\setminus\{i\}$. By Lemma~\ref{5coclique2}, we know that any maximal coclique containing $z_i, y_4$ and $y_5$ $(i=1,2,3)$ has exactly $4$ vertices and that $z_i$ is adjacent to all vertices of $\bar{\Delta}(y_i,y_4,y_5)$. By Lemma~\ref{5coclique3}, each of $y_1,y_2$ and $y_3$ lies in a clique $C_i$ of $\bar{\Delta}(z_i,y_4,y_5)$ with (at least) $11$ vertices. Note that $z_i$ and $z_j$ $(i\neq j\in\{1,2,3\})$ are adjacent to all vertices of $\bar{\Delta}(y_i,y_j,y_4,y_5)$, where $\bar{\Delta}(y_i,y_j,y_4,y_5)$ is a clique with $6$ vertices (see, Lemma~\ref{5coclique1} $(iii)$). As $z_i$ and $z_j$ $(i\neq j\in\{1,2,3\})$ have at least $6$ common neighbors, $z_1, z_2$ and $z_3$ are pairwise adjacent, and hence, none of $z_1,z_2$ and $z_3$ lies in any of $C_1,C_2$ and $C_3$.

Now we extend $C_i$ $(i=1,2,3)$ to a maximal clique $D_i$ inside $\bar{\Delta}(y_4,y_5)$. Assume that $\{i,j,k\}=\{1,2,3\}$. Note that $y_i$ is in $D_i$ but $y_j$ and $y_k$ are not in $D_i$. This means that $D_1,D_2$ and $D_3$ are all distinct and that $y_i$ has at most $5$ neighbors in each of $D_j$ and $D_k$. As $u_i$ is adjacent to all vertices of $\bar{\Delta}(y_i,y_4,y_5)$, $u_i$ has at least $6$ neighbors in each of $D_j$ and $D_k$, and hence $u_i$ is in $D_j\cap D_k$. So, $u_i$ is adjacent to all vertices in $C_j$ and $C_k$. This shows that $C_1,C_2$ and $C_3$ are disjoint as $C_i$ is a clique of $\bar{\Delta}(u_i,y_4,y_5)$. As $\bar{\Delta}(y_4,y_5)$ contains $C_1, C_2, C_3, z_1, z_2$ and $z_3$, $\bar{\Delta}(y_4,y_5)$ has at least $3\times 11+3=36$ vertices. This contradicts Lemma~\ref{5coclique1} $(v)$. Thus,  $\Sigma$ has no triangles. 

As $\Sigma$ is triangle-free and does not contain a coclique with $3$ vertices, it is easy to see that $\Sigma$ must be a pentagon. Without loss of generality, we may assume that  $1 \sim 2\sim 3 \sim 4 \sim 5$ and $1 \sim 5$. Let $u_i$ be a vertex of $\bar{\Delta}(y_{i-1}, y_i, y_{i+1})$ with valency 16 inside $\bar{\Delta}(y_{i-1}, y_i, y_{i+1})$ for $i =1,2,3,4,5$, where $y_0 = y_5$ and $y_6 = y_1$.  By Lemma~\ref{5coclique3} and the fact that $\{u_1, y_1, y_5\}$ is a coclique that can not be extended to a coclique with $5$ vertices, we see that $y_ 2$ lies in a clique $C_2'$ with $11$ vertices inside $\bar{\Delta}(u_1, y_1, y_5)$. In a similar fashion, we see that $y_2$ lies in a clique $C_2''$ with $11$ vertices inside $\bar{\Delta}(u_3, y_3, y_4)$.
As $\bar{\Delta}(u_1, y_1, y_5)$ and $\bar{\Delta}(u_3, y_3, y_4)$ has exactly 6 common vertices, namely the vertices of $\bar{\Delta}(y_1, y_3, y_4, y_5)$. That is, the induced subgraph on
$V(C_2') \cup V(C_2'')$ is complete. This means that $y_2$ lies in a clique $C_2$ with $16$ vertices.
This shows that $y_i$ lies in a clique $C_i$ with $16$ vertices for $i =1, 2, 3, 4, 5$.
But $u_4$ is adjacent to all vertices of $C_1$ and $C_2$. That is, there are two maximal cliques with at least 16 vertices that intersect. This contradicts Lemma \ref{clique}  $(iv)$. Thus, $\Delta$ does not contain a coclique with $5$ vertices.
\end{proof}

\section{No distance-regular graph with intersection array $\{80,54,12; 1,6,60\}$}
In this section we will first show that any local graph $\Delta$ of a distance-regular graph $\Gamma$ with intersection array $\{80,54,12; 1,6,60\}$ consists of $4$ disjoint cliques with 20 vertices, i.e., the graph $\Gamma$ is geometric. Then we will prove that there does not exist the distance-regular graph $\Gamma$.

By Proposition \ref{5coclique} and the fact that $3(25+1) < 80$ we find that all maximal cocliques in $\Delta$ have $4$ vertices. In the following lemma, for a given coclique in $\Delta$ with vertices $x_1,x_2,x_3$ and $x_4$, we consider $\bar{\Delta}(x_i,x_j,x_p)$, where $i,j,p\in\{1,2,3,4\}$.

\begin{lemma}\label{maxclique}
Let $\Gamma$ be a distance-regular graph with intersection array \\
$\{80,54,12; 1,6,60\}$. Let $x$ be a vertex of $\Gamma$ and let $\Delta:=\Delta(x)$ be the local graph at $x$. Let $\bar{C}$ be a coclique in $\Delta$ with distinct vertices $x_1, x_2, x_3$ and $x_4$. Then for $j, p, q\in\{1,2,3,4\}$, $\bar{\Delta}(x_j,x_p,x_q)$ is a clique with at least $11$ vertices containing $x_i$ and there is the unique maximal clique in $\Delta$ containing $\bar{\Delta}(x_j,x_p,x_q)$, where $\{i,j,p,q\}=\{1,2,3,4\}$.
\end{lemma}
\begin{proof}
It is easy to see that $\bar{\Delta}(x_j,x_p,x_q)$ has at least $11$ vertices containing $x_i$ as $x_i$ and $x_t$ has at most $5$ common neighbors in $\Delta$ for each $t\in\{j,p,q\}$. Also, any two vertices in $\bar{\Delta}(x_j,x_p,x_q)$ are adjacent otherwise there exists a coclique with $5$ vertices. So, $\bar{\Delta}(x_j,x_p,x_q)$ is a clique in $\Delta$ with at least $11$ vertices containing $x_i$. 

Now we consider vertices of $\Delta$ that are either in $\bar{\Delta}(x_j,x_p,x_q)$ or adjacent to all vertices in $\bar{\Delta}(x_j,x_p,x_q)$. Then any two vertices of them are adjacent as they have at least $11$ common neighbors, and hence those vertices induce a clique in $\Delta$ containing $\bar{\Delta}(x_j,x_p,x_q)$. Conversely, a vertex of a clique in $\Delta$ containing $\bar{\Delta}(x_j,x_p,x_q)$ is either in $\bar{\Delta}(x_j,x_p,x_q)$ or adjacent to all vertices in $\bar{\Delta}(x_j,x_p,x_q)$. So, the clique in $\Delta$ with vertex set $\{x\mid  x {\rm ~is~ adjacent~ to~ all~ vertices~ in~} \bar{\Delta}(x_j,x_p,x_q),{\rm~ or~} x \in \bar{\Delta}(x_j,x_p,x_q)\}$ is the unique maximal clique in $\Delta$ containing $\bar{\Delta}(x_j,x_p,x_q)$.
\end{proof}

We denote the unique maximal clique in $\Delta$ containing $\bar{\Delta}(x_j,x_p,x_q)$ by $D(x_i):=D_{\bar{C}}(x_i)$, where $\{i,j,p,q\}=\{1,2,3,4\}$ and  $x_1, x_2, x_3$ and $x_4$ are the vertices of a maximal coclique in $\Delta$. And we denote the number of vertices in $D(x_i)$ by $d(x_i):=d_{\bar{C}}(x_i)$ for $i=1,2,3,4$. By Lemma~\ref{maxclique}, we see that  $d(x_i)\geq11$ for $i=1,2,3,4$. Without loss of generality we assume that $d(x_1)\leq d(x_2)\leq d(x_3)\leq d(x_4)$.



 In the following lemma, for a given coclique in $\Delta$ with vertices $x_1,x_2,x_3$ and $x_4$, we study sizes of $D(x_i)'s$.

\begin{lemma}\label{disjoint1}  Let $\Gamma$ be a distance-regular graph with intersection array \\
$\{80,54,12; 1,6,60\}$. Let $x$ be a vertex of $\Gamma$ and let $\Delta:=\Delta(x)$ be the local graph at $x$. Let $\bar{C}$ be a coclique in $\Delta$ with distinct vertices $x_1, x_2, x_3$ and $x_4$. Then for $i\neq j\in \{1,2,3,4\}$, $\bar{\Delta}(x_i, x_j)$ can be partitioned into two cliques such that each of the two cliques has at least $11$ vertices. In particular, $D(x_i)$ and $D(x_j)$ are disjoint and $d(x_i)+d(x_j) \geq 28$ holds for $i\neq j\in\{1,2,3,4\}$.
\end{lemma}
\begin{proof}
Let $p\neq q\in \{1,2,3,4\}$ be numbers distinct from $i$ and $j$. Note that $\bar{\Delta}(x_i, x_j)$ includes both $\bar{\Delta}(x_i, x_j,x_p)$ and $\bar{\Delta}(x_i, x_j,x_q)$. By Lemma~\ref{maxclique},  both $\bar{\Delta}(x_i, x_j,x_p)$ and $\bar{\Delta}(x_i, x_j,x_q)$ are cliques with at least 11 vertices containing $x_q$ and $x_p$ respectively.  Define $C_p$ and $C_q$ as the maximal cliques in $\bar{\Delta}(x_i, x_j)$ containing $\bar{\Delta}(x_i, x_j,x_q)$ and $\bar{\Delta}(x_i, x_j,x_p)$ respectively. Note that by Lemma~\ref{clique}, $C_p$ and $C_q$ have at most $4$ common vertices. 

Now we consider a vertex $y$ in  $\bar{\Delta}(x_i, x_j)\setminus(\bar{\Delta}(x_i, x_j,x_p)\cup \bar{\Delta}(x_i, x_j,x_q))$ (if it exists). As $y$ has valency at least 15 in $\bar{\Delta}(x_i, x_j)$, we see that $y$ has at least $(15-4)/2 > 5$ neighbors in one of $\bar{\Delta}(x_i, x_j,x_p)$ and $\bar{\Delta}(x_i, x_j,x_q)$,  say without loss of generality $\bar{\Delta}(x_i, x_j,x_p)$, i.e., $y$ is adjacent to all vertices in $\bar{\Delta}(x_i, x_j,x_p)$. So, any vertex in $\bar{\Delta}(x_i, x_j)$ is contained in $C_p$ or $C_q$. If a vertex of $\bar{\Delta}(x_i,x_j)$ is contained in both $C_p$ and $C_q$, then the vertex has at least 27 neighbors in $\Delta$, a contradiction. So, no vertex in $\bar{\Delta}(x_i, x_j)$ is contained in both $C_p$ and $C_q$. This shows that $\bar{\Delta}(x_i, x_j)$ can be partitioned into two cliques $C_p$ and $C_q$, and both $C_p$ and $C_q$ have at least $11$ vertices. Note that $D(x_p)$ and $D(x_q)$ contain $C_p$ and $C_q$ respectively, i.e., $d(x_p)+d(x_q)\geq|C_p|+|C_q|\geq28$. If $D(x_p)$ and $D(x_q)$ have a common vertex, then the vertex has at least 27 neighbors in $\Delta$, a contradiction. So, $D(x_p)$ and $D(x_q)$ are disjoint. This finishes the proof.
\end{proof}

From Lemma~\ref{disjoint1}, we see that $D(x_1), D(x_2), D(x_3)$ and $D(x_4)$ are (pairwise) disjoint.

In the following lemma, we give some results on $d(x_1), d(x_2), d(x_3)$ and $d(x_4)$, where $x_1,x_2,x_3$ and $x_4$ are vertices of a coclique in $\Delta$. 

\begin{lemma}\label{disjoint2}  Let $\Gamma$ be a distance-regular graph with intersection array \\
$\{80,54,12; 1,6,60\}$. Let $x$ be a vertex of $\Gamma$ and let $\Delta:=\Delta(x)$ be the local graph at $x$. Let $\bar{C}$ be a coclique in $\Delta$ with distinct vertices $x_1, x_2, x_3$ and $x_4$. Then, without loss of generality, the following hold:
\begin{enumerate} 
\item[$(i)$] $11 \leq d(x_1) \leq d(x_2) \leq d(x_3) \leq d(x_4) \leq 20;$
\item[$(ii)$] $d(x_1) + d(x_2) + d(x_3) + d(x_4) \geq 64;$
\item[$(iii)$] $d(x_1) + d(x_2) + d(x_3) \geq 47;$
\item[$(iv)$] $d(x_1) + d(x_4) \geq 31.$
 \end{enumerate}
\end{lemma}

\begin{proof}
$(i)$ It follows from Lemma~\ref{clique} and Lemma \ref{disjoint1}.

\noindent
$(ii)$: As $\mu_{\Delta}(x_1, x_2) + \mu_{\Delta}(x_1, x_3) + \mu_{\Delta}(x_1, x_4) + \mu_{\Delta}(x_2, x_3) + \mu_{\Delta}(x_2, x_4) + \mu_{\Delta}(x_3, x_4) \geq
4\times(25+1) - 80 = 24,$ at least one of $\mu_{\Delta}(x_1, x_2) + \mu_{\Delta}(x_3, x_4)$, $\mu_{\Delta}(x_1, x_3) + \mu_{\Delta}(x_2, x_4)$ and $\mu_{\Delta}(x_1, x_4) + \mu_{\Delta}(x_2, x_3)$ is at least $8$. From Lemma~\ref{disjoint1}, we know that $d(x_i)+d(x_j)\geq 80-(2\times(25+1)-\mu_{\Delta}(x_p,x_q))$ holds for $\{i,j,p,q\}=\{1,2,3,4\}$. This means that $d(x_1) + d(x_2) + d(x_3) + d(x_4) \geq 2\times(80 - 52) + 8 = 64.$
 
\noindent 
$(iii)$: By Lemma~\ref{disjoint1}, we have $d(x_1) + d(x_2) \geq 28 +\mu_{\Delta}(x_3, x_4)$, $d(x_1) + d(x_3) \geq 28 + \mu_{\Delta}(x_2, x_4)$ and $d(x_2) + d(x_3) \geq 28 + \mu_{\Delta}(x_1, x_4)$.
As $ \mu_{\Delta}(x_1, x_4) + \mu_{\Delta}(x_2, x_4) + \mu_{\Delta}(x_3, x_4) \geq 24-(\mu_{\Delta}(x_1, x_2) + \mu_{\Delta}(x_1, x_3) + \mu_{\Delta}(x_2, x_3))\geq24 -15 =9$, we see that 
$2(d(x_1) + d(x_2) + d(x_3))\geq 84 +9=93$ and hence $(iii)$ follows.

\noindent
$(iv)$: As $\mu_{\Delta}(x_2, x_3) + \mu_{\Delta}(x_2, x_4) + \mu_{\Delta}(x_3, x_4) \geq 24-15 =9$, there exist distinct $y_1, y_2 \in \{ x_2, x_3, x_4\} $ such that $\mu_{\Delta}(y_1, y_2) \geq 3$. Let $y \in \{x_2, x_3, x_4\} \setminus \{ y_1, y_2\}$. Then $d(x_1) + d(x_4) \geq d(x_1) + d(y) \geq 28 + \mu_{\Delta}(y_1, y_2) \geq 31$, and hence $(iv)$ follows.
 
 \noindent
 This finishes the proof.
\end{proof}

We note that Lemma~\ref{disjoint2} $(iii)$ also means that $d(x_3)\geq16$.

 In the following lemma, we will show that if for a vertex $x$ of a coclique in $\Delta$ with 4 vertices, $d(x)=20$, then any local graph $\Delta$ of $\Gamma$ can be partitioned into $4$ cliques with $20$ vertices.

\begin{lemma}\label{disjoint3}  Let $\Gamma$ be a distance-regular graph with intersection array \\
$\{80,54,12; 1,6,60\}$. Let $x$ be a vertex of $\Gamma$ and let $\Delta:=\Delta(x)$ be the local graph at $x$. Let $\bar{C}$ be a coclique in $\Delta$ with distinct vertices $x_1,x_2,x_3$ and $x_4$. Without loss of generality, we asume that $d(x_1)\leq d(x_2)\leq d(x_3)\leq d(x_4)$. If $d(x_4) = 20$, then $d(x_1) = d(x_2) = d(x_3) = 20$ and $\Delta$ can be partitioned into four cliques with $20$ vertices.
\end{lemma}

\begin{proof}
If each vertex of $\Delta$ is contained in one of $D(x_1), D(x_2),D(x_3)$ and $D(x_4)$, then we are done (by Lemma~\ref{disjoint2}). So, we may assume that  there exists a vertex $z$ of $\Delta$ that is not contained in $D(x_i)$ for all $i=1,2,3,4$. Note that if $z$ has more than $5$ neighbors in $D(x_i)$ ($i=1,2,3,4$), then $D(x_i)$ is not maximal and $D(x_i)$ should contain $z$, as all vertices in $D(x_i)$ are adjacent to $z$, i.e., $z$ has at most $5$ neighbors in $D(x_i)$ for all $i=1,2,3,4$. Then clearly there exists a vertex $z_1$ in $D(x_1)$ that is not adjacent to $z$. As $D(x_2)$ has at least 11 vertices, there exists a vertex $z_2$ in $D(x_2)$ that is not adjacent to both $z$ and $z_1$. As $D(x_3)$ has at least $16$ vertices, we can extend the coclique $\{z, z_1, z_2\}$ to a coclique $\{z, z_1, z_2, z_3\}$, where $z_3$ is a vertex of $D(x_3)$. By Lemma~\ref{clique} $(ii)$ each of $z, z_1, z_2, z_3$ has exactly two neighbors in $D(x_4)$, and then we find a coclique $\{z, z_1, z_2, z_3, z_4\}$, where $z_4$ is a vertex of $D(x_4)$, a contradiction. This finishes the proof.
\end{proof}

 In the following proposition, we will prove that any local graph $\Delta$ of  $\Gamma$ consists of $4$ disjoint cliques with 20 vertices.

\begin{proposition}\label{four}
 Let $\Gamma$ be a distance-regular graph with intersection array \\
$\{80,54,12; 1,6,60\}$. Let $x$ be a vertex of $\Gamma$ and let $\Delta:=\Delta(x)$ be the local graph at $x$. Then the graph $\Delta$ contains $4$ disjoint cliques with $20$ vertices.
\end{proposition}
\begin{proof}
 Let $\bar{C}$ be a coclique in $\Delta$ with vertices $x_1,x_2,x_3$ and $x_4$. Without loss of generality, we assume that $d_{\bar{C}}(x_1)\leq d_{\bar{C}}(x_2)\leq d_{\bar{C}}(x_3)\leq d_{\bar{C}}(x_4)$. If each vertex of $\Delta$ is contained in one of $D_{\bar{C}}(x_1), D_{\bar{C}}(x_2),D_{\bar{C}}(x_3)$ and $D_{\bar{C}}(x_4)$, then we are done (by Lemma~\ref{disjoint2}). So, we may assume that there exists a vertex $z$ of $\Delta$ that is not contained in all of $D_{\bar{C}}(x_1),D_{\bar{C}}(x_2),D_{\bar{C}}(x_3)$ and $D_{\bar{C}}(x_4)$.  Note that $z$ has at most $5$ neighbors in each of $\bar{\Delta}(x_2,x_3,x_4)$ and $\bar{\Delta}(x_1,x_3,x_4)$  otherwise $z$ is in $D_{\bar{C}}(x_1)$ or $D_{\bar{C}}(x_2)$. Then we can find $z_1\in \bar{\Delta}(x_2,x_3,x_4)$ and $z_2\in \bar{\Delta}(x_1,x_3,x_4)$ such that $z,z_1$ and $z_2$ are mutually non-adjacent (as each of $\bar{\Delta}(x_1,x_3,x_4)$ and $\bar{\Delta}(x_2,x_3,x_4)$ has at least 11 vertices). Clearly, $z_1$ and $z_2$ are non-adjacent to $x_3$ and $x_4$, i.e., $\bar{C'}:=\{z_1,z_2,x_3, x_4\}$ is a maximal coclique in $\Delta$.

Now we consider four maximal cliques $D_{\bar{C'}}(z_1), D_{\bar{C'}}(z_2), D_{\bar{C'}}(x_3)$ and $D_{\bar{C'}}(x_4)$, but we do not assume that $d_{\bar{C'}}(z_1)\leq d_{\bar{C'}}(z_2)\leq d_{\bar{C'}}(x_3)\leq d_{\bar{C'}}(x_4)$. We will first show that $D_{\bar{C}}(x_1)=D_{\bar{C'}}(z_1)$ and $D_{\bar{C}}(x_2)=D_{\bar{C'}}(z_2)$ hold. By Lemma~\ref{disjoint1}, we know that $\bar{\Delta}(x_3,x_4)$ consists of two disjoint cliques $C_1$ and $C_2$, one of which contains $x_1$ and the other $x_2$. As $\bar{\Delta}(z_2, x_3, x_4)$ is a clique with at least $11$ vertices, without loss of generality we assume that $C_1$ contains at least $6$ of them. But this means that $\bar{\Delta}(z_2, x_3, x_4)$ is contained in $C_1$, and hence $D_{\bar{C}}(x_1)=D_{\bar{C'}}(z_1)$ holds. Similarly, we obtain $D_{\bar{C}}(x_2)=D_{\bar{C'}}(z_2)$.


Note that  $x_3$ is contained in $D_{\bar{C}}(x_3)\cap D_{\bar{C'}}(x_3)$. Then by Lemma~\ref{clique} $(iv)$, we find that $d_{\bar{C}}(x_3)+d_{\bar{C'}}(x_3)\leq30$. As $d_{\bar{C}}(x_3)\geq16$, we have $d_{\bar{C'}}(x_3)\leq14$. As $D_{\bar{C}}(x_1)=D_{\bar{C'}}(z_1)$ and  $D_{\bar{C}}(x_2)=D_{\bar{C'}}(z_2)$, we have $d_{\bar{C}}(x_1)=d_{\bar{C'}}(z_1)$ and $d_{\bar{C}}(x_2)=d_{\bar{C'}}(z_2)$. By Lemma~\ref{disjoint2} $(iii)$, we obtain $d_{\bar{C}}(x_1)+d_{\bar{C}}(x_2)+d_{\bar{C'}}(z_3)=d_{\bar{C'}}(z_1)+d_{\bar{C'}}(z_2)+d_{\bar{C'}}(z_3)\geq47$, and hence $d_{\bar{C}}(x_1)+d_{\bar{C}}(x_2)\geq33$ holds. Thus, we have $d_{\bar{C}}(x_3)\geq d_{\bar{C}}(x_2)\geq17$. As $d_{\bar{C}}(x_3)+d_{\bar{C'}}(x_3)\leq30$ and $d_{\bar{C}}(x_3)\geq17$, we obtain $d_{\bar{C'}}(x_3)\leq13$, and this shows that $d_{\bar{C}}(x_1)+d_{\bar{C}}(x_2)\geq34$. As $d_{\bar{C}}(x_2)\leq20$, we obtain $d_{\bar{C}}(x_1)\geq14$, and this shows that each of $d_{\bar{C'}}(z_1),d_{\bar{C'}}(z_2),d_{\bar{C'}}(x_3)$ and $d_{\bar{C'}}(x_4)$ is at least $14$. Then it contradicts $d_{\bar{C'}}(x_3)\leq13$. So, there does not exist a vertex $z$ of $\Delta$ that is not contained in all of $D_{\bar{C}}(x_1),D_{\bar{C}}(x_2),D_{\bar{C}}(x_3)$ and $D_{\bar{C}}(x_4)$. This finishes the proof.
\end{proof}

Proposition~~\ref{four} shows that the graph $\Gamma$ is a geometric distance-regular graph. Now, we prove that there does not exist such a geometric distance-regular graph by using a result of Bang and Koolen~\cite{KB2010}.

\begin{theorem}\label{main}
There does not exist a distance-regular graph with intersection array  $\{80, 54,12; 1, 6, 60\}$.
\end{theorem}

\begin{proof}
 Let $\Gamma$ be a distance-regular graph with intersection array \\ $\{80,54,12; 1,6,60\}$.
By Proposition \ref{four}, there exists a set of Delsarte cliques, i.e. cliques with 21 vertices, of $\Gamma$ such that each edge lies in exactly one such Delsarte clique. And this means that $\Gamma$ is geometric.  By Lemma~\ref{kbgeometric}, we know that $\tau_2 \geq \psi_1$ holds. But in this case, by Lemma~\ref{taupsi}, we have $\tau_2=2$ and $\psi_1=3$, a contradiction. This finishes the proof.
\end{proof}

\section*{Acknowledgments}
Q. Iqbal and M.U. Rehman are supported by Chinese Scholarship Council at USTC, Hefei, China. J.H. Koolen has been partially supported by the National Natural Science Foundation of China (Grants No. 11471009 and No. 11671376) and by the Anhui Initiative in Quantum Information Technologies (Grant No. AHY150200). J. Park is supported by Basic Research Program through the National Research Foundation of Korea funded by Ministry of Education (NRF-2017R1D1A1B03032016).


\begin{thebibliography}{99}
\bibitem{Bose63}
R.C. Bose, Strongly regular graphs, partial geometries and partially balanced designs, Pacific J. Math. {\bf13} (1963), 389–-419.

\bibitem{BCN89}
A.E. Brouwer, A.M. Cohen, A. Neumaier,  Distance-regular graphs. Springer-Verlag, Berlin, 1989.





\bibitem{ALG2010}
A.L. Gavrilyuk, On the Koolen-Park inequality and Terwilliger graphs. Electr. J. Combin. {\bf17} (2010), R125.

\bibitem{CG1993}
C.D. Godsil, Geometric distance-regular covers, New Zealand J. Math. {\bf22} (1993), 31--38.


\bibitem{KB2010}
J.H. Koolen and S. Bang, On distance-regular graphs with smallest eigenvalue at least $-m$.  J. Combin. Theory Ser. B. {\bf100} (2010), 573--584.

\bibitem{KPShilla}
J.H. Koolen and J. Park, Shilla distance-regular graphs. European J. Combin {\bf 31} (2010),  2064--2073.


\bibitem{VKT2016}
E.R. Van Dam, J. H. Koolen and H. Tanaka, Distance-regular graphs, Electr. J. Combin. (2016), $\#$DS22.
\end{thebibliography}
\end{document}